\documentclass[10pt]{article}

\usepackage{
	amsmath, 
	amssymb, 
	amsthm,
	fullpage, 
  hyperref,
}

\usepackage[english]{babel}

\newcommand{\CC}{\mathbb{C}}
\newcommand{\NN}{\mathbb{N}}
\newcommand{\ZZ}{\mathbb{Z}}
\newcommand{\RR}{\mathbb{R}}
\newcommand{\SU}{\mathrm{SU}}
\newcommand{\diag}{\mathrm{diag}}

\def\pFq#1#2#3#4#5{%
	{}_{#1}F_{#2}\biggl[\genfrac..{0pt}{}{#3}{#4};#5\biggr]%
}

\def\pfq#1#2#3#4#5#6{%
	{}_{#1}\phi_{#2}\biggl(\genfrac..{0pt}{}{#3}{#4};#5,#6\biggr)%
}

\theoremstyle{definition}
\newtheorem{theorem}{Theorem}[section]

\newtheorem{lemma}[theorem]{Lemma}

\newtheorem{example}[theorem]{Example}
\newtheorem{remark}[theorem]{Remark}
\newtheorem{corollary}[theorem]{Corollary}

\numberwithin{equation}{section} 

\begin{document}

\title{Explicit matrix inverses for lower triangular matrices with entries involving continuous $q$-ultraspherical polynomials}
\author{
  Noud Aldenhoven\footnote{Email address: \href{mailto:n.aldenhoven@math.ru.nl}{n.aldenhoven@math.ru.nl}}\\ 
  Radboud University Nijmegen, FNWI, IMAPP, \\
  Heyendaalseweg 135, 6525 AJ Nijmegen, The Netherlands
}
\date{}

\maketitle

\begin{abstract}
For a one-parameter family of lower triangular matrices with entries involving continuous \linebreak[4] $q$-ultraspherical polynomials we give an explicit lower triangular inverse matrix, with entries involving again continuous $q$-ultraspherical functions.
The matrices are $q$-analogues of results given by Cagliero and Koornwinder recently.
The proofs are not $q$-analogues of the Cagliero-Koornwinder case, but are of a different nature involving $q$-Racah polynomials.
Some applications of these new formulas are given.
Also the limit $\beta \to 0$ is studied and gives rise to continuous $q$-Hermite polynomials for $0 < q < 1$ and $q > 1$.
\end{abstract}

\section{Introduction}

In \cite{KvPR13} Koelink, van Pruijssen and Rom\'an needed to invert a lower triangular matrix with entries involving Gegenbauer (or ultraspherical) polynomials.
The solution was given by Cagliero and Koornwinder \cite{CK14} in the wider context of a two-parameter family of lower triangular matrices involving Jacobi polynomials.
The inverse of this matrix is given in terms of Jacobi polynomials as well.
Cagliero and Koornwinder \cite{CK14} solved this problem using the Rodrigues formula for the Jacobi polynomials and some variations on the product rule.
Thereafter Koelink, de los R\'ios and Rom\'an \cite{KdlRR14} used the results of Cagliero and Koornwinder \cite{CK14} with an extra free parameter.

In this paper we give a partial $q$-analogue of the result of Cagliero and Koornwinder \cite{CK14}.
In a forthcoming paper \cite{AKR14}, which is a quantum analogue of \cite{KvPR12, KvPR13}, the main Theorem \ref{thm:main} is used to obtain an inverse of a lower triangular matrix with entries involving continuous $q$-ultraspherical polynomials.
Theorem \ref{thm:main} gives the inverse of this matrix in a more general situation.
Theorem \ref{thm:main} is the main result of this paper.

\begin{theorem} \label{thm:main}
Let $\beta \in \CC \backslash \{0\}$, $\beta \neq q^{\frac{k}{2}}$ for $k \in \ZZ$. 
Define doubly infinite lower triangular matrices $L^{\beta}(x)$ and $M^{\beta}(x)$ by
\begin{align*}
L^{\beta}(x)_{m, n} &= \frac{1}{(\beta^2 q^{2n}; q)_{m - n}} C_{m - n}(x;\beta q^n|q), \qquad n \leq m, \\[0.3cm]
M^{\beta}(x)_{m, n} &= \frac{\beta^{m-n}q^{(m-1)(m-n)}}{(\beta^2 q^{m+n-1}; q)_{m - n}} C_{m-n}(x;\beta^{-1}q^{1-m}|q), \qquad n \leq m,
\end{align*}
where $m, n \in \ZZ$ and $C_m(x;\beta|q)$ are the continuous $q$-ultraspherical polynomials defined in Section \ref{sec:prelim} for all $\beta$. 
Then $M^{\beta}(x)$ and $L^{\beta}(x)$ are each other's inverse, i.e. $L^{\beta}(x)M^{\beta}(x) = I = M^{\beta}(x)L^{\beta}(x)$, where $I_{m, n} = \delta_{m, n}$ is the identity.
\end{theorem}

The proof of Theorem \ref{thm:main} is given in Section \ref{sec:mainproof}.

Theorem \ref{thm:main} has a finite dimensional analogue, because the entries of $L^{\beta} M^{\beta}$ only involve finite sums of continuous $q$-ultraspherical polynomials.
From Theorem \ref{thm:main} we have to following corollary.

\begin{corollary} \label{cor:main}
For a non-negative integer $N$ and $\beta \in \CC \backslash \{0\}$ such that $\beta \neq q^{-\frac{k}{2}}$ for $k=0,1,\ldots,2N-2$.
Define lower triangular matrices $L^{\beta}(x)$ and $M^{\beta}(x)$
\begin{align*}
L^{\beta}(x)_{m, n} &= \frac{1}{(\beta^2 q^{2n}; q)_{m - n}} C_{m - n}(x;\beta q^n|q), \qquad 0 \leq n \leq m \leq N \\[0.3cm]
M^{\beta}(x)_{m, n} &= \frac{\beta^{m-n}q^{(m-1)(m-n)}}{(\beta^2 q^{m+n-1}; q)_{m - n}} C_{m-n}(x;\beta^{-1}q^{1-m}|q), \qquad 0 \leq n \leq m \leq N.
\end{align*}
Then $M^{\beta}(x)$ and $L^{\beta}(x)$ are each others inverse, i.e. $L^{\beta}(x)M^{\beta}(x) = I = M^{\beta}(x)L^{\beta}(x)$, where $I$ is the identity matrix.
\end{corollary}

The proof of Theorem \ref{thm:main} is not a straightforward $q$-analogue of the proof given by Cagliero and Koornwinder \cite{CK14}.
The proof uses $q$-Racah polynomials and does not use Rodrigues formulas or product rules of differentials which are the essential ingredients for the proof in \cite{CK14}.
In particular, the $q \to 1$ limit of the proof presented here gives an alternative proof of the special case $\alpha = \beta$ of Cagliero and Koornwinder \cite{CK14}.

We compute the coefficients of $e^{ik\theta}$ of products of two continuous $q$-ultraspherical polynomials and express the coefficients in terms of terminating balanced basic hypergeometric series ${}_4 \phi_3$.
For certain parameters this series transforms to a $q$-Racah polynomial.
The orthogonality relations of the $q$-Racah polynomials then lead to Theorem \ref{thm:main}.
The proof of Theorem \ref{thm:main}, for $q \to 1$, gives an interesting new proof of \cite[Theorem 4.1]{CK14} in the special case $\alpha = \beta$, showing that the coefficients of $e^{ik\theta}$ of products of certain Gegenbauer polynomials are actually Racah polynomials.
The entries of the matrix identity $L(x)M(x) = I$ in \cite[Theorem 4.1]{CK14} correspond to orthogonality relations of Racah polynomials, see Example \ref{exp:qto1}.

In Section \ref{sec:bto0} we study matrices $L^{\beta}$ and $M^{\beta}$ for Theorem \ref{thm:main} for a suitable limit $\beta \to 0$.
The entries of $L^{\beta}$ become continuous $q$-Hermite polynomials and the entries of $M^{\beta}$ converge to continuous $q^{-1}$-Hermite polynomials as $\beta \to 0$.

We emphasise again that our proof is different than the proof of Calgiero and Koornwinder \cite{CK14}.
It is possible to extend the proof of Calgiero and Koornwinder \cite{CK14} to a $q$-analogue for certain polynomials in the $q$-Askey scheme \cite{KLS10}.
For example \cite[Lemma 5.1]{CK14} has a $q$-analogue for the $q$-derivative operator \cite[Exercise 1.12]{GR04}.
Then with the use of Rodrigues' formula and suitable parameters for the orthogonal polynomials it is possible to find $q$-analogues for \cite[(4.1),(4.2)]{CK14}.
The author was able to extend \cite[(4.1), (4.2)]{CK14} to the little $q$-Jacobi polynomials.
However these results involve different $q$-shifts in the $x$ of the polynomials and don't seem to lead to a result similar to Theorem \ref{thm:main} or \cite[Theorem 4.1]{CK14}.
Also Calgiero and Koornwinder \cite{CK14} were motivated by \cite{BG09, KvPR13} to extend their formulas to a two parameter family of Jacobi polynomials.
We lack this motivation and therefore decided not to include these results for the little $q$-Jacobi polynomials in this paper. 
We didn't extend the results to other families of polynomials.

\section{Preliminaries} \label{sec:prelim}
We recall some facts on basic hypergeometric series and related polynomials, see Gasper and Rahman \cite{GR04} and Koekoek, Lesky and Swarttouw \cite{KLS10}.
We fix $0 < q < 1$ and we follow notation of \cite{GR04}.

For $\beta \in \CC$, the continuous $q$-ultraspherical polynomials are given by
\begin{align}
C_n(x; \beta | q) &= \sum_{k = 0}^{n}
  \frac{
    (\beta; q)_{k} (\beta; q)_{n - k}
  }{
    (q; q)_{k} (q; q)_{n - k}
  } e^{i (n - 2k) \theta}, \qquad x = \cos(\theta),
  \label{eqn:qultraexp}
\end{align}
see \cite[Exercise 1.28]{GR04} and \cite[\S 14.10]{KLS10}.
Notice that the continuous $q$-ultraspherical polynomials are defined for all $\beta$.
A generating function for the continuous $q$-ultraspherical polynomials is
\begin{align}
\sum_{n = 0}^{\infty} C_n(x; \beta | q) t^n
  &= \frac{
    (\beta t e^{i\theta}, \beta t e^{-i\theta}; q)_{\infty}
  }{
    (te^{i\theta}, te^{-i\theta}; q)_{\infty}
  } \qquad |t| < 1, \qquad x = \cos(\theta) \in [-1, 1], \label{eqn:q-ultra-gen}
\end{align}
see \cite[Exercise 1.29]{GR04} and \cite[(14.10.27)]{KLS10}.

For $\alpha, \beta, \gamma, \delta \in \RR$ such that $q \alpha = q^{-N}$, $\beta\delta q = q^{-N}$ or $\gamma q = q^{-N}$, for $N \in \NN$, define the $q$-Racah polynomials 
\begin{align} \label{eqn:qracah}
R_n(\mu(x);\alpha,\beta,\gamma,\delta;q) = \pfq{4}{3}{
  q^{-n}, \alpha \beta q^{n+1}, q^{-x}, \gamma \delta q^{x+1}
}{
  \alpha q, \beta \delta q, \gamma q
}{q}{q}, \qquad \mu(x) = q^{-x} + \gamma \delta q^{x+1},
\end{align}
where $n = 0, 1, \ldots, N$.
If $q \alpha = q^{-N}$ and $\beta = 1$ the $q$-Racah polynomials are not orthogonal with respect to a positive measure.
Still the $q$-Racah polynomials are orthogonal 
\begin{align*} 
\sum_{x=0}^N \frac{
  (q^{-N}, \gamma \delta q; q)_x
}{
  (q, \gamma \delta q^{N+2}; q)_x
}
\frac{
  (1 - \gamma \delta q^{2x + 1})
}{
  (1 - \gamma \delta q)
} q^{Nx} R_m(\mu(x);q^{-N-1},1,\gamma,\delta;q) &R_n(\mu(x);q^{-N-1},1,\gamma,\delta;q) \nonumber \\
&= \delta_{m,n} h_m(\gamma, \delta; N),
\end{align*}
where $h_m(\gamma, \delta; N)$ is given in \cite[\S 7.2]{GR04} and \cite[\S 14.2]{KLS10}.
It follows that if $n = 0$ we have
\begin{align} 
\sum_{x=0}^N \frac{
  (q^{-N}, \gamma \delta q; q)_x
}{
  (q, \gamma \delta q^{N+2}; q)_x
}
\frac{
  (1 - \gamma \delta q^{2x + 1})
}{
  (1 - \gamma \delta q)
} q^{Nx} R_m(\mu(x);q^{-N-1},1,\gamma,\delta;q) = \delta_{m,0} h_0(\gamma, \delta; N), \label{eqn:qracah_ortho}
\end{align}
where $h_0(\gamma, \delta; N) = \delta_{N, 0}$ if $\gamma \neq q^{-\ell}$ and $\delta \neq q^{-m}$ with $\ell, m = 1, 2, \ldots, N$.

Note that (\ref{eqn:qracah_ortho}) can also be proved directly, also see \cite{AW79}.
To show this substitute (\ref{eqn:qracah}) in (\ref{eqn:qracah_ortho}) so that
\begin{align*}
\sum_{x=0}^N \frac{
  (q^{-N}, \gamma \delta q; q)_x
}{
  (q, \gamma \delta q^{N+2}; q)_x
}
\frac{
  (1 - \gamma \delta q^{2x + 1})
}{
  (1 - \gamma \delta q)
} q^{Nx}
\sum_{k=0}^{N}
\frac{
  (q^{-m}, q^{m-N}, q^{-x}, \gamma\delta q^{x+1}; q)_{k}
}{
  (q, q^{-N}, \delta q, \gamma q; q)_{k}
} q^k.
\end{align*}
Then expand the left hand side of (\ref{eqn:qracah_ortho}) in $q^x$ observing that it is a polynomial in $q^{x}$ of degree $N-k$.
Finally applying the summation formula \cite[(II.21)]{GR04} on the $x$-sum gives the right hand side of (\ref{eqn:qracah_ortho}). 

\begin{remark} \label{rmk:referee}
One of the referees pointed out that if $q \alpha = q^{-N}$ and $\beta = 1$ then from (\ref{eqn:qracah}) it follows that $R_n = R_{N-n}$.
Therefore for $n > \frac{1}{2}N$ the polynomial $R_n$ will have degree $N - n < n$.
So there can be no non-degenerate orthogonality.
However, the system of polynomials $R_n$ for $n \leq \frac{1}{2}N$ can still be orthogonal with respect to positive weights.
\end{remark}

Sears' transformation formula, \cite[(III.15) \& (III.16)]{GR04}, for terminating balanced ${}_4 \phi_3$ series is
\begin{align} 
\pfq{4}{3}{
  q^{-n}, a, b, c
}{
  d, e, f
}{q}{q}
&= a^n \frac{
  (ea^{-1},fa^{-1};q)_n
}{
  (e,f;q)_n
} \pfq{4}{3}{
  q^{-n}, a, db^{-1}, dc^{-1}
}{
  d, aq^{1-n}e^{-1}, aq^{1-n}f^{-1}
}{q}{q} \label{eqn:sears1} \\
&= \frac{
  (a, ef(ab)^{-1}, ef(ac)^{-1}; q)_{n}
}{
  (e, f, ef(abc)^{-1}; q)_{n}
} \pfq{4}{3}{
  q^{-n}, ea^{-1}, fa^{-1}, ef(abc)^{-1}
}{
  ef(ab)^{-1}, ef(ac)^{-1}, q^{1-n}a^{-1}
}{q}{q}, \label{eqn:sears2}
\end{align}
where $abc = defq^{n-1}$.

\section{Proof of Theorem \texorpdfstring{\ref{thm:main}}{1.1}} \label{sec:mainproof}

The idea of the proof of Theorem \ref{thm:main} is to first expand a sum of products of continuous $q$-ultraspherical polynomials in terms of $e^{ik \theta}$, where $x = \cos(\theta)$. 
We show that the coefficients of $e^{ik\theta}$ are balanced basic hypergeometric series ${}_4\phi_3$.
For the continuous $q$-ultraspherical polynomials with parameters as in Theorem \ref{thm:main} we show that the coefficients of $e^{ik\theta}$ correspond to the orthogonality relations for $q$-Racah polynomials.
This proves the key Lemma \ref{lem:Linv} from which Theorem \ref{thm:main} follows.

\begin{lemma} \label{lem:technical}
Take $n \in \NN$.
Let $\alpha_k, \beta_k$ and $c(k)$ be constants for $k=0,1,\ldots,n$.
Then
\begin{align} \label{eqn:lem1.1}
\sum_{k=0}^{n} c(k)\, C_{n-k}(x;\alpha_k|q) C_k(x;\beta_k|q) 
  = \sum_{p=0}^{n} d(p)\, e^{i(n-2p)\theta}, \qquad x = \cos(\theta),
\end{align}
where $d(p)$ is given by
\begin{align} 
d(p) &= \sum_{k=0}^{n-p} c(k)\, 
  \frac{(\alpha_k;q)_p}{(q;q)_{p}}
  \frac{(\alpha_k;q)_{n-p-k}}{(q;q)_{n-p-k}}
  \frac{(\beta_k;q)_{k}}{(q;q)_{k}}
  \pfq{4}{3}{
    q^{-p}, \alpha_k q^{n-p-k}, q^{-k}, \beta_k
  }{
    q^{1-p}\alpha_k^{-1}, q^{1-k}\beta_k^{-1}, q^{n-p-k+1}
  }{q}{q^2 \alpha_k^{-1}\beta_k^{-1}} \nonumber \\
&+ \sum_{k=n-p+1}^{n} c(k)\,
  \frac{(\alpha_k;q)_{n-k}}{(q;q)_{n-k}}
  \frac{(\beta_k;q)_{p-n+k}}{(q;q)_{p-n+k}}
  \frac{(\beta_k;q)_{n-p}}{(q;q)_{n-p}}
  \pfq{4}{3}{
    q^{k-n}, q^{p-n}, \alpha_k, \beta_k q^{k+p-n}
  }{
    q^{1-n+k}\alpha_k^{-1}, q^{k+p-n+1}, q^{1-n+p}\beta_k^{-1}
  }{q}{q^2 \alpha_k^{-1} \beta_k^{-1}}. \label{eqn:d(p)}
\end{align}
\end{lemma}

\begin{proof}
First expand the left hand side of (\ref{eqn:lem1.1}) using (\ref{eqn:qultraexp}), so that the left hand side of (\ref{eqn:lem1.1}) equals
\begin{align} \label{eqn:lem1.2}
\sum_{k = 0}^n c(k)\,
  \sum_{s = 0}^{n-k} \frac{(\alpha_k;q)_s}{(q;q)_s} \frac{(\alpha_k;q)_{n-k-s}}{(q;q)_{n-k-s}}
  \sum_{t = 0}^{k} \frac{(\beta_k;q)_{t}}{(q;q)_t} \frac{(\beta_k;q)_{k-t}}{(q;q)_{k-t}}
  e^{i(n-2(s+t))\theta}.
\end{align}
Now fix $p = s+t$ and substitute $s = p-t$ in (\ref{eqn:lem1.2}) so that the coefficient of $e^{i(n-2p)\theta}$ becomes
\begin{align} \label{eqn:lem1.3}
\sum_{k=0}^{n} c(k)\, \sum_{t=0 \vee (k+p-n)}^{k \wedge p}
  \frac{(\alpha_k;q)_{p-t}}{(q;q)_{p-t}}
  \frac{(\alpha_k;q)_{n-k-p+t}}{(q;q)_{n-k-p+t}}
  \frac{(\beta_k;q)_{t}}{(q;q)_{t}}
  \frac{(\beta_k;q)_{k-t}}{(q;q)_{k-t}}.
\end{align}
For $0 \leq k \leq n-p$ so that $k+p-n \leq 0$, the $t$-sum of (\ref{eqn:lem1.3}) is, after simplifying the $q$-Pochhammer symbols, the balanced terminating ${}_4 \phi_3$
\begin{align} \label{eqn:lem1.4}
\frac{(\alpha_k;q)_p}{(q;q)_{p}} 
\frac{(\alpha_k;q)_{n-p-k}}{(q;q)_{n-p-k}}
\frac{(\beta_k;q)_{k}}{(q;q)_{k}}
\pfq{4}{3}{
  q^{-p}, \alpha_k q^{n-p-k}, q^{-k}, \beta_k
}{
  q^{1-p}\alpha_k^{-1}, q^{1-k}\beta_k^{-1}, q^{n-p-k+1}
}{q}{q^2 \alpha_k^{-1}\beta_k^{-1}}.
\end{align}
For $n-p \leq k \leq n$ so that $k+p-n \geq 0$ substitute $t \mapsto t + k + p - n$ so that the $t$-sum of (\ref{eqn:lem1.3}) is, after simplifying the $q$-Pochhammer symbols, the balanced terminating ${}_4 \phi_3$
\begin{align} \label{eqn:lem1.5}
\frac{(\alpha_k;q)_{n-k}}{(q;q)_{n-k}}
\frac{(\beta_k;q)_{p-n+k}}{(q;q)_{p-n+k}}
\frac{(\beta_k;q)_{n-p}}{(q;q)_{n-p}}
\pfq{4}{3}{
  q^{k-n}, q^{p-n}, \alpha_k, \beta_k q^{k+p-n}
}{
  q^{1-n+k}\alpha_k^{-1}, q^{k+p-n+1}, q^{1-n+p}\beta_k^{-1}
}{q}{q^2 \alpha_k^{-1} \beta_k^{-1}}.
\end{align}
Combining (\ref{eqn:lem1.4}) and (\ref{eqn:lem1.5}) gives (\ref{eqn:d(p)}).
\end{proof}

\begin{remark}
Since the continuous $q$-ultraspherical polynomials are polynomials in $x$ the coefficients of $e^{i(n-p)\theta}$ and $e^{ip\theta}$ of the left hand side of (\ref{eqn:lem1.1}) must be equal.
Therefore $d(p) = d(n-p)$ and (\ref{eqn:lem1.1}) can be rewritten in terms of Chebychev polynomials $T_p$ of the first kind, see \cite[\S 9.8.2]{KLS10}, as follows
\begin{align*}
\sum_{k=0}^{n} c(k)\, C_{n-k}(x;\alpha_k|q) C_k(x;\beta_k|q) 
  = \sum_{p=0}^{[\frac{n}{2}]} (2 - \delta_{n, 2p}) d(p) T_{n-2p}(x).
\end{align*}
\end{remark}

\begin{remark}
It is possible to write (\ref{eqn:d(p)}) uniformly
\begin{align*}
d(p) &= \sum_{k=0}^{n} c(k)\, 
  \frac{(\alpha_k;q)_p}{(q;q)_{p}}
  \frac{(\beta_k;q)_{k}}{(q;q)_{k}}
  \frac{(\alpha_k;q)_{\infty}}{(\alpha_k q^{n-k-p};q)_{\infty}}
  \frac{(q^{n-k-p+1};q)_{\infty}}{(q;q)_{\infty}} \\
  &\qquad \times \pfq{4}{3}{
    q^{-p}, \alpha_k q^{n-p-k}, q^{-k}, \beta_k
  }{
    q^{1-p}\alpha_k^{-1}, q^{1-k}\beta_k^{-1}, q^{n-p-k+1}
  }{q}{q^2 \alpha_k^{-1}\beta_k^{-1}}.
\end{align*}
If $0 \leq k \leq n - p$ we have that (\ref{eqn:lem1.4}) is equal to
\begin{align} \label{eqn:formald(p)}
\frac{(\alpha_k;q)_p}{(q;q)_{p}}
\frac{(\beta_k;q)_{k}}{(q;q)_{k}}
\frac{(\alpha_k;q)_{\infty}}{(\alpha_k q^{n-k-p};q)_{\infty}}
\frac{(q^{n-k-p+1};q)_{\infty}}{(q;q)_{\infty}}
\pfq{4}{3}{
  q^{-p}, \alpha_k q^{n-p-k}, q^{-k}, \beta_k
}{
  q^{1-p}\alpha_k^{-1}, q^{1-k}\beta_k^{-1}, q^{n-p-k+1}
}{q}{q^2 \alpha_k^{-1}\beta_k^{-1}}.
\end{align}
Use the convention
\begin{align*}
\frac{(q^{1-N}; q)_{\infty}}{(q^{1-N}; q)_{t}} = (q^{1-N+t};q)_{\infty},
\end{align*}
so that for $n-p < k \leq n$ 
\begin{align*}
\frac{
  (q^{1-N}; q)_{\infty}
}{
  (q; q)_{\infty}
}
\sum_{t=0}^{\infty}
\frac{
  C_t
}{
  (q, q^{1-N}; q)_{t}
}
=
\frac{
  (q^{N+1}; q)_{\infty}
}{
  (q; q)_{\infty}
}
\sum_{t=0}^{\infty}
\frac{
  C_{N+t}
}{
  (q, q^{N+1}; q)_{t}
},
\end{align*}
where $N \in \NN$ and $C_t$ are arbitrary constants.
Then for $N = k + p - n$ we have that (\ref{eqn:formald(p)}) becomes
\begin{align}
&\frac{(\alpha_k;q)_{p}}{(q;q)_p}
\frac{(\beta_k;q)_{k}}{(q;q)_k}
\frac{(\alpha_k;q)_{\infty}}{(\alpha_k q^{n-k-p};q)_{\infty}}
\frac{(q^{1+k+p-n};q)_{\infty}}{(q;q)_{\infty}} \nonumber \\
&\qquad \times \sum_{t=0}^{\infty}
\frac{
  (q^{-p}, \alpha_k q^{n-k-p}, \beta_k, q^{-k}; q)_{t+k+p-n}
}{
  (q, q^{1+k+p-n};q)_{t} (q^{1-p}\alpha_k^{-1}, q^{1-k}\beta_k^{-1}; q)_{t+k+p-n}
} \left(q^2 \alpha_k^{-1} \beta_k^{-1}\right)^{t+k+p-n} \nonumber \\
&=
\frac{(\alpha_k;q)_{p}}{(q;q)_p}
\frac{(\beta_k;q)_{k}}{(q;q)_k}
\frac{(\alpha_k;q)_{\infty}}{(\alpha_kq^{n-k-p};q)_{\infty}}
\frac{(q^{1+k+p-n};q)_{\infty}}{(q;q)_{\infty}}
\frac{
  (q^{-p}, \alpha_k q^{n-k-p}, \beta_k, q^{-k}; q)_{k+p-n}
}{
  (q^{1-p}\alpha_k^{-1}, q^{1-k}\beta_k^{-1}; q)_{k+p-n}
}
\left(
q^2 \alpha_k^{-1} \beta_k^{-1}
\right)^{k+p-n} \nonumber \\
&\qquad \times \pfq{4}{3}{
  q^{k-n}, \alpha_k, \beta_k q^{k+p-n}, q^{p-n}
}{
  q^{k+p-n+1}, q^{k-n+1} \alpha_k^{-1}, q^{1+p-n} \beta_k^{-1}
}{q}{q^2\alpha_k^{-1}\beta_k^{-1}}. \label{eqn:formald(p).2}
\end{align}
Simplifying the $q$-Pochhammer symbols of (\ref{eqn:formald(p).2}) shows that (\ref{eqn:formald(p).2}) is equal to (\ref{eqn:lem1.5}).
\end{remark}

\begin{lemma} \label{lem:Linv}
For $m, n \in \ZZ$ such that $n \leq m$.
Let $\beta \in \CC$ such that $\beta^2 \neq q^{-2m+1}, q^{-2m+2}, \ldots, q^{-2n}$. 
Then
\begin{align} \label{eqn:lem2.1}
\sum_{k = 0}^{m - n}
  \frac{
    (1 - \beta^2 q^{2n + 2k - 1} )
  }{
    (\beta^2 q^{2n + k - 1}; q)_{m - n + 1}
  } \beta^k q^{k(k + n - 1)}
  C_{m - n - k}(x;\beta q^{k}| q)
  C_{k}(x; \beta^{-1} q^{1 - k - n}| q) = \delta_{m, n}.
\end{align}
\end{lemma}

\begin{proof}
Apply Lemma \ref{lem:technical} with $n, \alpha_k, \beta_k$ specialised to $m-n, q^{k+n}\beta, q^{1-k-n}\beta^{-1}$ so that in particular $\alpha_k \beta_k = q$ for all $k$.
Then the left hand side of (\ref{eqn:lem2.1}) is $\sum_{p=0}^{m-n} d(p) e^{i(m-n-2p)\theta}$, where $x = \cos(\theta)$ and
\begin{align} 
d(p) &= \sum_{k=0}^{m-n-p}
  \frac{
    (1 - \beta^2 q^{2n + 2k - 1} )
  }{
    (\beta^2 q^{2n + k - 1}; q)_{m - n + 1}
  } \beta^k q^{k(k + n - 1)} \nonumber \\
  &\qquad \times \frac{(\beta q^{k+n};q)_{p}}{(q;q)_{p}}
  \frac{(\beta q^{k+n};q)_{m-n-p-k}}{(q;q)_{m-n-p-k}}
  \frac{(\beta^{-1}q^{1-k-n};q)_{k}}{(q;q)_k} \nonumber \\
  &\qquad \qquad \times \pfq{4}{3}{
    q^{-k}, q^{-p}, \beta q^{m - p}, q^{-n-k+1} \beta^{-1}
  }{
    \beta q^{n}, q^{-p - n - k + 1}\beta^{-1}, q^{m - n - k - p + 1}
  }{q}{q} \nonumber \\
  &+ \sum_{k=m-n-p+1}^{m-n}
    \frac{
      (1 - \beta^2 q^{2n + 2k - 1} )
    }{
      (\beta^2 q^{2n + k - 1}; q)_{m - n + 1}
    } \beta^k q^{k(k + n - 1)} \nonumber \\
    &\qquad \times \frac{(q^{k+n}\beta;q)_{m-n-k}}{(q;q)_{m-n-k}}
    \frac{(q^{1-k-n}\beta^{-1};q)_{p-m+n+k}}{(q;q)_{p-m+n+k}}
    \frac{(q^{1-k-n}\beta^{-1};q)_{m-n-p}}{(q;q)_{m-n-p}} \nonumber \\
    &\qquad \qquad \times \pfq{4}{3}{
      q^{k-m+n}, q^{p-m+n}, q^{k+n}\beta, q^{1+p-m}\beta^{-1}
    }{
      q^{1-m}\beta^{-1}, q^{k+p-m+n+1}, \beta q^{2n-m+p+k}
    }{q}{q}.
  \label{eqn:lem2.2}
\end{align}
We transform the basic hypergeometric series ${}_4 \phi_3$ of (\ref{eqn:lem2.2}).
Apply Sears' transformation formula (\ref{eqn:sears1}) to the first ${}_4\phi_3$ in (\ref{eqn:lem2.2}) to see that the ${}_4\phi_3$ is equal to
\begin{align}
\frac{
  (q^{-n - k + 1}\beta^{-1}, q^{m - n - k + 1}; q)_{k}
}{
  (q^{-n - k - p + 1}\beta^{-1}, q^{m - n - k - p + 1}; q)_{k}
} q^{-pk}
\pfq{4}{3}{
  q^{-k}, q^{-p}, q^{n - m + p}, \beta^2 q^{2n + k - 1} 
}{
  \beta q^{n}, \beta q^{n}, q^{n - m}
}{q}{q}. \label{eqn:lem2.case1}
\end{align}
Apply Sears' transformation formula (\ref{eqn:sears2}) to the second ${}_4 \phi_3$ in (\ref{eqn:lem2.2}) in order to see that the ${}_4\phi_3$ is equal to
\begin{align}
\frac{
  (q^{1+p-m} \beta^{-1}, q^{n-m}, \beta q^n; q)_{m-n-p}
}{
  (\beta q^{2n - m + k + p}, q^{1-m} \beta^{-1}, q^{-k}; q)_{m - n - p}
} \pfq{4}{3}{
  q^{n-m+p}, \beta^2 q^{2n + k - 1}, q^{-p}, q^{-k}
}{
  q^{n-m}, \beta q^n, \beta q^n
}{q}{q}.
\label{eqn:lem2.case2}
\end{align}
The ${}_4 \phi_3$ of (\ref{eqn:lem2.case1}) and (\ref{eqn:lem2.case2}) can be written as the $q$-Racah polynomial $R_p(\mu(k); q^{n - m - 1}, 1, \beta q^{n - 1}, \beta q^{n - 1}; q)$, see (\ref{eqn:qracah}).
Therefore (\ref{eqn:lem2.2}) becomes, after simplifying the $q$-Pochhammer symbols using $(q^r\beta^{-1};q)_{\ell} = (-1)^{\ell} q^{\frac{1}{2}\ell(\ell-1) + r\ell}\beta^{-\ell} (\beta q^{1-r-\ell};q)_{\ell}$ repeatedly,
\begin{align}
\frac{
  (\beta q^{n}; q)_{p} (\beta q^{n}; q)_{m - n - p}
}{
  (\beta^2 q^{2n}; q)_{m - n} (q; q)_{p} (q; q)_{m - n - p}
} &\sum_{k = 0}^{m - n}
\frac{
  (\beta^2 q^{2n - 1}, q^{n - m}; q)_{k}
}{
  (q, \beta^2 q^{n + m}; q)_{k}
}
\frac{
  (1 - \beta^2 q^{2n + 2k - 1})
}{
  (1 - \beta^2 q^{2n - 1})
}
q^{k(m - n)} \nonumber \\
&\qquad \times R_{p}(\mu(k); q^{n - m - 1}, 1, \beta q^{n - 1}, \beta q^{n - 1}; q). 
\label{eqn:lem2.4}
\end{align}
The $k$-sum of (\ref{eqn:lem2.4}) corresponds to the orthogonality relations (\ref{eqn:qracah_ortho}) for the $q$-Racah polynomial.
Hence (\ref{eqn:lem2.4}) becomes
\begin{align*}
d(p) = \frac{
  ( \beta q^{n}; q)_{m - n}
}{
  ( \beta q^{2n}; q)_{m - n} (q; q)_{m - n}
} \delta_{p, 0} h_0(\beta q^n, \beta q^n; m-n).
\end{align*}
Since $h_0(\beta q^n, \beta q^n; m-n) = 0$ if $n < m$ and $h_0(\beta q^n, \beta q^n; m-n) = 1$ if $m = n$, the result follows.
\end{proof}

\begin{proof}[Proof of Theorem \ref{thm:main}]
Multiplying the matrices $L^{\beta}$ and $M^{\beta}$ it is sufficient to evaluate the entries of $L^{\beta} M^{\beta}$ for $m \geq n$.
Hence
\begin{align*}
(L^{\beta}(x) M^{\beta}(x))_{m, n}
  &= \sum_{k = n}^{m}
  \frac{
    \beta^{k - n} q^{(k - 1)(k - n)}
  }{
    ( \beta^2 q^{2k}; q)_{m - k} ( \beta^2 q^{n + k - 1}; q)_{k - n}
  } 
  C_{m - k}(x; \beta q^{n} | q) 
  C_{k - n}(x; q^{1-k} \beta^{-1} | q) \\
&= \sum_{k = 0}^{m - n}
  \frac{
    (1 - \beta^2 q^{2n + 2k - 1} )
  }{
    (\beta^2 q^{2n + k - 1} ; q)_{m - n + 1}
  } \beta^k q^{k(k + n - 1)}
  C_{m - n - k}(x; \beta q^{k + n}| q)
  C_{k}(x; q^{1-k-n} \beta^{-1}| q).
\end{align*}
Applying Lemma \ref{lem:Linv} then yields the result.
\end{proof}

\section{Applications}
\begin{example} \label{exp:qto1}
The limit $q \to 1$ in the proof of Theorem \ref{thm:main} gives a new proof for \cite[Theorem 4.1]{CK14} for $\alpha = \beta$.
Lemma \ref{lem:technical} gives
\begin{align*}
\sum_{k=0}^n c(k) C^{(\alpha_k)}_{n-k}(x) C^{(\beta_k)}_{k}(x) = \sum_{p=0}^n d(p) e^{i(n-2p)\theta}, \qquad x = \cos(\theta),
\end{align*}
where $C^{(\alpha)}_{k}(x)$ are the Gegenbauer polynomials, see \cite[\S 9.8.1]{KLS10}, and
\begin{align*}
d(p) &= \sum_{k=0}^{n-p} 
  c(k)
  \frac{(\alpha_k)_{p}}{p!}
  \frac{(\alpha_k)_{n-p-k}}{(n-p-k)!} 
  \frac{(\beta_k)_{k}}{k!} 
  \pFq{4}{3}{
    -p, \alpha_k + n - p - k, -k, \beta_k
  }{
    1-p-\alpha_k, 1-k-\beta_k, n-p-k+1
  }{2-\alpha_k-\beta_k} \\
  &\qquad + \sum_{k=n-p+1}^{n}
    c(k)
    \frac{(\alpha_k)_{n-k}}{(n-k)!}
    \frac{(\beta_k)_{p-n+k}}{(p-n+k)!}
    \frac{(\beta_k)_{n-p}}{(n-p)!} \\
  &\qquad \qquad \times \pFq{4}{3}{
    k-n, p-n, \alpha_k, \beta_k+k+p-n
  }{
    1-n+k-\alpha_k, k+p-n+1, 1-n+p-\beta_k
  }{2-\alpha_k-\beta_k}.
\end{align*}
Then Lemma \ref{lem:Linv} yields, for $0 \leq n \leq m$ and $\alpha \in \CC$, $2\alpha \neq -2m+1, -2m+2 \ldots, -2n$,
\begin{align*}
\sum_{k=0}^{m-n} \frac{(2n+2k+2\alpha-1)}{(2n+k+2\alpha-1)_{m-n+1}} C_{m-n-k}^{(\alpha+k+n)}(x) C_{k}^{(1-k-n-\alpha)}(x) = \delta_{m,n},
\end{align*}
which is the key equation to show \cite[Theorem 4.1]{CK14} for the case $\alpha = \beta$.
\end{example}

\begin{example}
The problem of finding an inverse of the matrix $L^{\beta}$ in Theorem \ref{thm:main} originally arose in \cite{AKR14} where the finite dimensional lower triangular matrix
\begin{align*}
L(x)_{m,n} = q^{m-n}
  \frac{
    (q^2;q^2)_{m} (q^2;q^2)_{2n+1}
  }{
    (q^2;q^2)_{m+n+1} (q^2;q^2)_{n}
  } C_{m-n}(x;q^{2n+2}|q^2), \qquad 0 \leq n \leq m \leq N,
\end{align*}
for arbitrary $N \in \NN$ appears.
Using Corollary \ref{cor:main} in base $q^2$ with $\beta = q^2$ after conjugation with a diagonal matrix we find that the inverse matrix is given by
\begin{align*}
\left(L(x)\right)^{-1}_{m,n} = q^{(2m+1)(m-n)}
  \frac{
    (q^2;q^2)_{m} (q^2;q^2)_{m+n}
  }{
    (q^2;q^2)_{2m} (q^2;q^2)_{n}
  } C_{m-n}(x;q^{-2m}|q^2), \qquad 0 \leq n \leq m \leq N.
\end{align*}
Note that the entries of $L(x)$ and its inverse $L(x)^{-1}$ are independent of the size of $N$.
\end{example}

\begin{example}
From the generating function (\ref{eqn:q-ultra-gen}) for the continuous $q$-ultraspherical polynomials it follows that
\begin{align*}
\sum_{n = 0}^{\infty} C_n(x; \alpha \beta | q) t^n
  &= \frac{
    (\alpha t e^{i\theta}, \alpha t e^{-i\theta}; q)_{\infty}
  }{
    (t e^{i \theta}, t e^{-i \theta}; q)_{\infty}
  } 
  \frac{
    (\alpha \beta t e^{i \theta}, \alpha \beta t e^{-i \theta}; q)_{\infty}
  }{
    (\alpha t e^{i \theta}, \alpha t e^{-i \theta}; q)_{\infty}
  }
  = \sum_{m, n = 0}^{\infty}
    C_m(x; \alpha | q) C_n(x; \beta | q) t^m (\alpha t)^n.
\end{align*}
Comparing the powers of $t$ shows
\begin{align*}
C_n(x; \alpha \beta | q) 
  &= \sum_{k = 0}^n \alpha^k C_{n - k}(x; \alpha | q) C_k(x; \beta | q).
\end{align*}
Now take $\beta = \alpha^{-1}$, then (\ref{eqn:qultraexp}) for $\beta = 1$ gives
\begin{align} \label{eqn:gen1}
\delta_{n, 0} 
  = \sum_{k = 0}^n \alpha^k C_{n - k}(x; \alpha | q) C_k(x; \alpha^{-1} | q).
\end{align}
On the other hand from Lemma \ref{lem:technical} it follows that
\begin{align} \label{eqn:gen2}
\sum_{k=0}^n \alpha^k C_{n-k}(x;\alpha|q) C_k(x;\alpha^{-1}|q) = \sum_{p=0}^n d(p) e^{i(n-2p)\theta}, \qquad x = \cos(\theta).
\end{align}
Combining (\ref{eqn:gen1}) and (\ref{eqn:gen2}) it follows that $d(p) = \delta_{n,0}$.
Writing out the explicit expression of $d(p)$ gives for $n > 0$ the identity
\begin{align*}
0 &= \sum_{k=0}^{n-p} \alpha^{k}
  \frac{(\alpha;q)_p}{(q;q)_p}
  \frac{(\alpha;q)_{n-p-k}}{(q;q)_{n-p-k}}
  \frac{(\alpha^{-1};q)_{k}}{(q;q)_{k}}
  \pfq{4}{3}{
    q^{-p}, \alpha q^{n-p-k}, q^{-k}, \alpha^{-1}
  }{
    q^{1-p}\alpha^{-1}, \alpha q^{1-k}, q^{n-p-k+1}
  }{q}{q^2} \\
  &\qquad + \sum_{k=n-p+1}^{n} \alpha^{k}
    \frac{(\alpha;q)_{n-k}}{(q;q)_{n-k}}
    \frac{(\alpha^{-1};q)_{p-n+k}}{(q;q)_{p-n+k}}
    \frac{(\alpha^{-1};q)_{n-p}}{(q;q)_{n-p}}
    \pfq{4}{3}{
      q^{k-n}, q^{p-n}, \alpha, q^{k+p-n}\alpha^{-1}
    }{
      q^{1-n+k}\alpha^{-1}, q^{k+p-n+1}, \alpha q^{1-n+p}
    }{q}{q^2}.
\end{align*}
In particular if $p=0$
\begin{align} \label{eqn:app3}
\sum_{k=0}^{n} \alpha^k \frac{(\alpha;q)_{n-k}}{(q;q)_{n-k}} \frac{(\alpha^{-1};q)_{k}}{(q;q)_{k}} = 0.
\end{align}
Remark that (\ref{eqn:app3}) also follows from the $q$-Chu-Vandermonde sum \cite[(1.5.2)]{GR04}.
For $p=1$ and $n \mapsto n+1$ we find
\begin{align*}
\sum_{k=0}^n \alpha^k
\frac{ (\alpha; q)_{n-k} }{ (q;q)_{n-k} }
\frac{ (\alpha^{-1}; q)_{k} }{ (q; q)_{k} }
\left(
1 + \frac{
  (1 - \alpha q^{n-k})(1 - q^k)
}{
  (1 - \alpha q^{1-k})(1 - q^{n+1-k})
} q^{1-k}
\right)
=
\alpha^n \frac{(\alpha^{-1}; q)_{n}}{(q;q)_{n}}.
\end{align*}
Remark that this result also follows from applying (\ref{eqn:app3}) twice.
\end{example}

\section{Limit case \texorpdfstring{$\beta \to 0$}{beta to one}} \label{sec:bto0}
Define $L^0(x)$ and $M^0(x)$ by $L^0(x)_{m,n} = \lim_{\beta \to 0} L^{\beta}(x)_{m,n}$ and $M^0(x)_{m,n} = \lim_{\beta \to 0} M^{\beta}(x)_{m,n}$ where the limit is taken over $\beta \neq q^{\frac{k}{2}}$, where $k \in \ZZ$.
We show that the limits exist, that the entries of $L^0(x)$ are given in terms of continuous $q$-Hermite polynomials and that the entries of $M^0(x)$ are a given in terms of continuous $q^{-1}$-Hermite polynomials.

The continuous $q$-Hermite polynomials are given by
\begin{align} \label{eqn:qhermite}
H(x|q) = \sum_{k=0}^{n}
  \frac{(q;q)_n}{(q;q)_k(q;q)_{n-k}} e^{i(n-2k)\theta}, \qquad x = \cos(\theta),
\end{align}
see \cite[\S 14.26]{KLS10}.
The continuous $q$-Hermite polynomials are, apart from a different normalisation, the special case $\beta = 0$ of the continuous $q$-ultraspherical polynomials
\begin{align} \label{eqn:limCnHn}
C_n(x;0|q) = \frac{H_n(x|q)}{(q;q)_n}. 
\end{align}
The corresponding generating function for the continuous $q$-Hermite polynomials is
\begin{align} \label{eqn:Hngenfunc}
\sum_{n=0}^{\infty} \frac{H_n(x|q)}{(q;q)_n} t^n
= 
\frac{1}{(te^{i\theta}, te^{-i\theta}; q)_{\infty}} 
\qquad |t| < 1, \qquad x = \cos(\theta),
\end{align}
see \cite[(14.26.11)]{KLS10}.

The polynomials $H_n(x|q^{-1})$ are called the continuous $q^{-1}$-Hermite polynomials and are defined by taking $q \mapsto q^{-1}$ in (\ref{eqn:qhermite}), see \cite{Ask89}.
The continuous $q$-Hermite polynomials are orthogonal with respect to a positive measure on $(-1, 1)$.
However the continuous $q^{-1}$-Hermite polynomials are orthogonal on the imaginary axis and correspond to an indeterminate moment problem, see \cite{Ask89} and \cite{IM94}.

\begin{theorem} \label{thm:main0}
The doubly infinite lower triangular matrices $L^{0}(x)$ and $M^{0}(x)$ are given by
\begin{align*}
L^{0}(x)_{m, n} &= \frac{H_{m-n}(x|q)}{(q;q)_{m-n}}, \qquad
M^{0}(x)_{m, n} = (-1)^{m-n} q^{\binom{m-n}{2}} \frac{H_{m-n}(x|q^{-1})}{(q;q)_{m-n}}, \qquad n \leq m,
\end{align*}
where $m, n \in \ZZ$. 
$M^{0}(x)$ and $L^{0}(x)$ are each others inverse, i.e. $L^{0}(x)M^{0}(x) = I = M^{0}(x)L^{0}(x)$, where $I_{m, n} = \delta_{m, n}$ is the identity.
\end{theorem}

\begin{proof}
With (\ref{eqn:limCnHn}) we have for $n \leq m$
\begin{align*}
L^{0}(x)_{m,n} 
  = \lim_{\beta \to 0} L^{\beta}(x)_{m,n}
  = \lim_{\beta \to 0} \frac{1}{(\beta^2 q^{2n}; q)_{m-n}} C_{m-n}(x;\beta q^n|q)
  = \frac{H_{m-n}(x|q)}{(q;q)_{m-n}}.
\end{align*}
From (\ref{eqn:qultraexp}) it follows that $C_n(x;\beta|q) = (\beta q^{-1})^n C_n(x;\beta^{-1}|q^{-1})$.
Therefore write $M^{\beta}(x)_{m,n}$ as
\begin{align*}
\frac{\beta^{m-n} q^{(m-1)(m-n)}}{(\beta^2 q^{m+n-1}; q)_{m-n}} 
  (\beta^{-1} q^{-m})^{m-n} 
  C_{m-n}(x;\beta q^{m-1}|q^{-1})
=
  \frac{
    q^{-(m-n)}
  }{
    (\beta^2 q^{m+n-1}; q)_{m-n}
  } C_{m-n}(x; \beta q^{m-1}|q^{-1}).
\end{align*}
Upon taking the limit $\beta \to 0$ and using (\ref{eqn:limCnHn}) we find
\begin{align*}
M^{0}(x)_{m,n}
  = \lim_{\beta \to 0} M^{\beta}(x)_{m,n}
  = q^{-(m-n)} \frac{
    H_{m-n}(x|q^{-1})
  }{
    (q^{-1};q^{-1})_{m-n}
  }
  =
  (-1)^{m-n} q^{\binom{m-n}{2}}
  \frac{H_{m-n}(x|q^{-1})}{(q;q)_{m-n}}.
\end{align*}
From Theorem \ref{thm:main} it follows that $L^0(x)M^0(x) = I = M^0(x)L^0(x)$.
\end{proof}

\begin{corollary} \label{cor:main0}
For $N \in \NN$ define lower triangular matrices $L^{0}(x)$ and $M^{0}(x)$
\begin{align*}
L^{0}(x)_{m, n} &= \frac{H_{m-n}(x|q)}{(q;q)_{m-n}}, \qquad
M^{0}(x)_{m, n} = (-1)^{m-n} q^{\binom{m-n}{2}} \frac{H_{m-n}(x|q^{-1})}{(q;q)_{m-n}}, \qquad 0 \leq n \leq m \leq N,
\end{align*}
Then $M^{0}(x)$ and $L^{0}(x)$ are each others inverse, i.e. $L^{0}(x)M^{0}(x) = I = M^{0}(x)L^{0}(x)$, where $I$ is the identity matrix.
\end{corollary}

\begin{remark}
Theorem \ref{thm:main0} also follows from a generating function for the continuous $q^{-1}$-Hermite polynomials.
From \cite[Theorem 21.2.1]{Ism09}
\begin{align} \label{eqn:Hnq-1genfunc}
\sum_{n=0}^{\infty} (-1)^n q^{\binom{n}{2}} \frac{H_n(x|q^{-1})}{(q;q)_n} t^n
  = (te^{i\theta}, te^{-i\theta};q)_{\infty}, \qquad |t| < 1, \qquad x = \cos(\theta).
\end{align}
Combining (\ref{eqn:Hngenfunc}) and (\ref{eqn:Hnq-1genfunc}) it follows that for $|t| < 1$
\begin{align*}
1 = \frac{(te^{i\theta}, te^{-i\theta};q)_{\infty}}{(te^{i\theta}, te^{-i\theta}; q)_{\infty}}
  &= \left(
    \sum_{m=0}^{\infty} \frac{H_m(x|q)}{(q;q)_{m}} t^m
  \right)
  \left(
    \sum_{n=0}^{\infty} (-1)^n q^{\binom{n}{2}} \frac{H_n(x|q^{-1})}{(q;q)_{n}} t^n
  \right) \\
  &= \sum_{p=0}^{\infty} \left(
    \sum_{k=0}^p \frac{H_{p-k}(x|q)}{(q;q)_{p-k}} (-1)^k q^{\binom{k}{2}} \frac{H_k(x|q^{-1})}{(q;q)_k}
  \right) t^p.
\end{align*}
Take $p = m-n$ so that we have
\begin{align} \label{eqn:main0alt}
\sum_{k=0}^{m-n} \frac{H_{m-n-k}(x|q)}{(q;q)_{m-n-k}} (-1)^k q^{\binom{k}{2}} \frac{H_k(x|q^{-1})}{(q;q)_k} = \delta_{m,n}.
\end{align}
From (\ref{eqn:main0alt}) Theorem \ref{thm:main0} also follows.
\end{remark}

\subsection*{Acknowledgements}
The author thanks Erik Koelink and Pablo Rom\'an for many useful discussions.
Parts of the paper have been discussed with Tom Koornwinder, the author thanks him for his input.

The research of the author is supported by the Netherlands Organisation for Scientific Research (NWO) under project number \textbf{613.001.005} and by the Belgian Interuniversity Attraction Pole Dygest \textbf{P07/18}.

The author thanks the referees for their input, pointing out Remark \ref{rmk:referee} and deriving (\ref{eqn:lem2.case2}) in an alternative way.

\end{document}